\documentclass{amsart}

 \usepackage{amsmath,amssymb,amsthm,texdraw,accents}

 \numberwithin{equation}{section}

 \newtheorem{theorem}[equation]{Theorem}
 
 \newtheorem{lemma}[equation]{Lemma}
 
 \theoremstyle{definition}
 \newtheorem{definition}[equation]{Definition}
 \theoremstyle{remark}
 \newtheorem{remark}[equation]{Remark}
 \newtheorem{example}[equation]{Example}

 \newcommand{\nc}{\newcommand}

 \nc{\la}{\lambda}
 \nc{\La}{\Lambda}
 \nc{\pwl}{P^+}
 \nc{\fit}{\tilde{f}_i}
 \nc{\eit}{\tilde{e}_i}
 \nc{\veps}{\varepsilon}
 \nc{\vphi}{\varphi}

 \nc{\NM}{\mathcal{M}} % set of Nakajima monomials
 \nc{\CB}{B}
 \nc{\Bla}{\CB(\lambda)}
 \nc{\Bcup}{\CB(\cup)}
 \nc{\Binf}{\CB(\infty)}
 \nc{\Tb}{T}
 \nc{\Sla}{\Tb(\lambda)}
 \nc{\SML}{\Tb(\infty)}
 \nc{\Rt}{R}
 \nc{\Rla}{\Rt(\lambda)}
 \nc{\RML}{\Rt(\infty)}
 \nc{\RL}{\Rt^L}
 \nc{\Rcup}{\Rt(\cup)}
 \nc{\Mo}{M}
 \nc{\Mla}{\Mo(\lambda)}
 \nc{\No}{N}
 \nc{\Nla}{\No(\lambda)}
 \nc{\Tv}{\mathsf{T}}
 \nc{\Tvla}{\Tv_{\la}}
 \nc{\tvla}{\mathsf{t}_{\la}}
 \nc{\x}{\bf x}

 \nc{\hwvT}{T_{\infty}}
 \nc{\hwvR}{R_{\infty}}

 \nc{\ur}{r}
 \nc{\ut}{t}

 \nc{\bpl}{\bar{\pi}_{\lambda}}
 \nc{\bpm}{\bar{\pi}_{\mu}}
 \nc{\Z}{\mathbf{Z}}
 \nc{\bsim}{\accentset{\alpha}{\sim}}
 \nc{\msim}{\accentset{\beta}{\sim}}

 \DeclareMathOperator{\wt}{wt}

\begin{document}

\title
[Crystal $\Bla$ in $B(\infty)$ for $G_2$ type Lie Algebra]
{Crystal $\Bla$ in $B(\infty)$ for $G_2$ type Lie Algebra}

\author{Min Kyu Kim}
\address{%
  Department of Mathematics Education, Gyeongin National University of Education, 45 Gyodae-Gil, Gyeyang-gu, Incheon, 407-753, Republic of Korea}
\email{mkkim@kias.re.kr}

\author{Hyeonmi Lee}
\address{%
  Department of Mathematics and Research Institute for Natural Sciences, Hanyang University, Seoul 133, KOREA}
\email{hyeonmi@hanyang.ac.kr}

\thanks{%
  The first author was supported by GINUE research fund.
  The second author was supported by Basic Science Research Program through the National Research Foundation of Korea(NRF) funded by the Ministry of Education, Science and Technology (NRF-2009-0068934).}


\begin{abstract}
A previous work gave a combinatorial description of the crystal $B(\infty)$, in terms of certain simple Young tableaux referred to as the marginally large tableaux, for finite dimensional simple Lie algebras.
Using this result, we present an explicit description of the crystal~$B(\lambda)$, in terms of the marginally large tableaux, for the $G_2$ Lie algebra type.
We also provide a new description of $B(\la)$, in terms of Nakajima monomials, that is in natural correspondence with our tableau description.
\end{abstract}

\maketitle

%%%%%%%%%%%%%%%%%%%%%%
\section{Introduction}

Quantum group $U_q(\mathfrak{g})$ is a $q$-deformation of the universal enveloping algebra~$U(\mathfrak{g})$ over a Lie algebra~$\mathfrak{g}$~\cite{De85,Ji85}, and the crystal base~$\Binf$ exposes the structure of the subalgebra $U_q^-(\mathfrak{g})$ of the quantum group~\cite{Ka91}.
In this work, we provide a new expression for the crystal base~$\Bla$ of the irreducible highest weight module of highest weight $\la$, using elements of $B(\infty)$, for the $G_2$ Lie algebra type.
We obtain an explicit realization of~$B(\la)$ by presenting it essentially as a specific subset of the marginally large tableaux realization of~$B(\infty)$.
During this process we develop and present another description of~$\Bla$ in terms of Nakajima monomials, which is in natural correspondence with our tableau description.

Our starting point is the strict embedding $\Omega_{\la}:B(\la)\hookrightarrow B(\infty)\otimes \Tvla$, introduced in~\cite{Naka}, where $\Tvla$ is a certain single-element crystal.
The existence of this embedding implies that the crystal $B(\infty)\otimes\mathsf{T}_{\lambda}$ contains a copy of the irreducible highest weight crystal $B(\lambda)$ as a sub-crystal.
After replacing the abstract object $B(\infty)$ with its more concrete realization $T(\infty)$~\cite{HoLe}, consisting of all marginally large tableaux, we carefully list the tableaux that should belong to the sub-crystal $B(\lambda)$, and claim the resulting set as a concrete realization for $B(\lambda)$.

Note that the tableau realization $T(\infty)$ we utilize was obtained by combining an interpretation of the crystal~$\Binf$ as a certain union of the highest weight crystals $B(\lambda)$'s~\cite{HoLe2,Ka02} and the tableau realization of $B(\la)$, for the $G_2$ Lie algebra type, given by Kang and Misra~\cite{KaMi}.
In all, we are obtaining a tableau realization of $B(\lambda)$ as a subset of a tableau realization for $B(\infty)$, which is a collection of elements taken from the tableau realizations of $B(\lambda)$'s.
However, our resulting tableau realization of $B(\lambda)$ is quite different from the initial tableau realization of $B(\lambda)$.
In fact, crystal $T(\infty)$ of marginally large tableaux is a collection of only the simplest tableaux appearing in the realizations for the $B(\lambda)$'s, with none of them involving the complicated, so called, distance conditions, so that our description of $B(\la)$ also contains only the simple tableaux.
Our marginally large tableau description of $B(\la)$ is almost as simple as the well-known semi-standard tableau realization of $B(\la)$ for the special linear Lie algebra type.

To verify that our list of tableaux, carefully selected from $T(\infty)\otimes \Tvla$, gives a realization of $B(\lambda)$, we return to the strict embedding $\Omega_{\la}: B(\la)\hookrightarrow B(\infty)\otimes \Tvla$ and show that our set is precisely the image of this map.
Once again, since dealing with the abstract object $B(\lambda)$ would be difficult, we develop a concrete realization $N(\la)$ of $B(\la)$, consisting of Nakajima monomials, so that the map $N(\la) \hookrightarrow T(\infty)\otimes \Tvla$ could be computed very explicitly.

Let us briefly discuss our monomial realization $N(\la)$ of~$B(\la)$.
Nakajima introduced a set of monomials~\cite{na03} on which Kashiwara gave a crystal structure and showed certain connected components of the crystal to be isomorphic to the crystals~$\Bla$~\cite{Ka03}.
Various descriptions of~$\Bla$ for the $G_2$ Lie algebra type based on this result have already appeared~\cite{jkks,sh}, but these were not suitable for our purpose of computing the map $N(\la) \hookrightarrow T(\infty)\otimes \Tvla$.
Our new monomial realization $N(\la)$ has the property that the distance of each element from the highest weight monomial, in terms of Kashiwara operator actions, is directly accessible from the expression of the element itself.
This property allows natural correspondences to be made between monomials and marginally large tableaux, so that the map $N(\la) \hookrightarrow T(\infty)\otimes \Tvla$ can be handled explicitly.

As one application of this work, we expect to be able to give a combinatorial description of the Casselman-Shalika formula for the $G_2$ type.
The works~\cite{LeSa2,LeSa3} introduced a combinatorial rule for expressing the Gindikin-Karpelevich formula associated to crystal $B(\infty)$ for the classical finite and $G_2$ Lie algebra types, using the set of marginally large tableaux.
The Casselman-Shalika formula is a companion formula to this formula, associated to the highest weight crystal $B(\la)$, and the work~\cite{LeSa1} also gave a corresponding combinatorial rule for the $A_n$ Lie algebra type, using the well-known semi-standard tableaux.
The collection of marginally large tableaux expressing the crystal $B(\la)$, introduced in the current work, should allow an analogous treatment for the $G_2$ case.

The rest of this paper is organized as follows.
In Section~\ref{sec:2}, we recall the basic theory of Nakajima monomials and the tableau description of~$B(\infty)$.
Our monomial description $N(\la)$ of $B(\la)$ is presented in Section~\ref{sec:3}.
The description of crystal~$\Bla$ in terms of marginally large tableaux, which is our main result, is obtained in the final section.
The section also contains the correspondence between our monomial and tableau descriptions.

%%%%%%%%%%%%%%%%%%%%%%%%%%%%%%%%%%%%
\section{Preliminaries}\label{sec:2}

In this section, we recall the basics concerning the Nakajima monomials~\cite{Ka03} and the description of~$\Binf$ in terms of marginally large tableaux~\cite{HoLe}.
Even though there are more general results, our discussion in this work is restricted to the $G_2$ Lie algebra type.

We will use standard notation found in the textbook~\cite{HoKan} and assume knowledge of the related notions, such as the following: index set~$I=\{1,2\}$, simple root~$\alpha_i$, coroot~$h_i$, fundamental weight~$\La_i$, set of dominant integral weights~$\pwl$, Cartan matrix $(\alpha_i(h_j))_{i,j\in I}$ with $\alpha_2(h_1)=-3$ and $\alpha_1(h_2)=-1$, quantum group~$U_q(G_2)$, abstract crystal with associated Kashiwara operators $\eit$, $\fit$ and maps~$\wt$, $\veps_i$, $\vphi_i$, irreducible highest weight crystal~$\Bla$, tensor product rule, negative part~$U_q^-(G_2)$ of~$U_q(G_2)$, and crystal basis~$\Binf$ of~$U_q^-(G_2)$.

%%%%%%%%%%%%%%%%%%%%%%%%%%%%%%%%%%%%%%%%%%%%%%%%%%%%%%%%%%%%%%%%%
\subsection{Nakajima monomials and crystal $\Bla$}\label{sec:2-1}

Let us recall the basics of the {Nakajima monomial} theory from~\cite{Ka03,na03}.
The set of \emph{Nakajima monomials} in the variables $Y_i(m)$, where $i\in I=\{1,2\}$ and $m\in\Z$, is denoted by $\NM$.
Each monomial is of the form $\prod_{(i,m)} {Y_i(m)}^{y_i(m)}$, with nonzero exponent $y_i(m)\in\Z$ appearing at only finitely many $(i,m)\in I\times\Z$.

The crystal structure on the set $\NM$, as defined by Kashiwara~\cite{Ka03}, is given below.
For each $m\in\Z$, fix the notation
\begin{align*}
U_1(m) &={Y_1(m)}{Y_1(m+1)}{Y_2(m)}^{-1},\\
U_2(m) &={Y_2(m)}{Y_2(m+1)}{Y_1(m+1)}^{-3}.
\end{align*}
For every monomial $N=\prod_{(i,m)}{Y_i(m)}^{y_i(m)}\in\NM$ and $i\in I$, we set
\begin{align*}
\wt(N) &= \sum_{i\in I}\big(\sum_{m\in\Z} y_i(m)\big)\La_i,\\
\vphi_i(N) &= \max\big\{\sum_{k\leqslant m} y_i(k) \,\vert\, m\in\Z\big\},\\
\veps_i(N) &= \max\big\{-\sum_{k>m} y_i(k) \,\vert\, m\in\Z\big\}.
\end{align*}
To prepare for the definition of the Kashiwara operator applications, we introduce the values
\begin{align}\label{mfme}
m_f(N,i) &= \min\big\{m \mid \vphi_i(N) = \sum_{k\leqslant m} y_i(k) \big\},\\
m_e(N,i) &= \max\big\{m \mid \veps_i(N) = -\sum_{k>m} y_i(k) \big\}.
\end{align}
The Kashiwara operator actions are given by
\begin{align*}
\fit(N)&=
\begin{cases}
0    &\text{for $\vphi_i(N)=0$,}\\
U_i(m_f(N,i))^{-1}N    &\text{for $\vphi_i(N)>0$,}
\end{cases}\\
\eit(N)&=
\begin{cases}
0    &\text{for $\veps_i(N)=0$,}\\
U_i(m_e(N,i))N     &\text{for $\veps_i(N)>0$.}
\end{cases}
\end{align*}

The following theorem from~\cite{Ka03} gives a realization of the irreducible highest weight crystal.

\begin{theorem}\label{nakakash}
For a maximal vector $N\in\NM$, the connected component of the crystal~$\NM$ containing~$N$ is isomorphic to $\CB(\wt(N))$.
\end{theorem}

Kashiwara~\cite{Ka03} gave multiple crystal structures on the set~$\NM$ and the above realization theorem holds true for each of these crystal structures, but we will deal only with the crystal structure described above.

%%%%%%%%%%%%%%%%%%%%%%%%%%%%%%%%%%%%%%%%%%%%%%%%%%%%%%%%%%%%%%%%%%%%%%
\subsection{Tableau description of crystal $B(\infty)$}\label{sec:2-2}

A description of $\Binf$, given in terms of marginally large tableaux, was presented by~\cite{HoLe}.
This was a subset of the union of tableau descriptions for the crystals $\Bla$, which were given by~\cite{KaMi}.
Our new presentation of $\Bla$ will be based on this marginally large tableau description of~$\Binf$.
The rest of this subsection presents material from~\cite{HoLe}.

For $\la\in P^+$, we denote by $T(\la)$ the set of tableaux given in~\cite{KaMi}, which is a description for~$B(\la)$.
The set of alphabets to be used inside the boxes constituting a tableau of $T(\la)$ will be denoted by $J$, and it will be equipped with an ordering $\prec$, as given
in~\cite{KaMi}:
\begin{equation*}
J=\{1\prec{2}\prec{3}\prec{0}\prec\bar{3}\prec\bar{2}\prec\bar{1}\}.
\end{equation*}

Marginally large tableaux, recalled in the next definition, are among the simplest tableaux appearing in the tableau realization $T(\la)$, and are almost as simple as the tableaux for the special linear Lie algebra type.
The top row of a tableau will be referred to as its \emph{first} row throughout this work.

\begin{definition}
For $i\in I$, a \emph{basic $i$-column} is a single column of $i$-many boxes, with the box at its $k$-th row containing the entry~$k$, for each $1\leqslant k\leqslant i$.
A semi-standard tableau (w.r.t. $\prec$) consisting of $2$ rows with entries from $J$ is \emph{marginally large}, if
\begin{enumerate}
\item it contains exactly one basic $i$-column, for each $i\in I$,
\item only $2$ and $3$ appear as box entries in the second row, and
\item $0$ appears at most once as an entry in the first row.
\end{enumerate}
The set of all marginally large tableaux is denoted by~$\SML$.
\end{definition}

\begin{remark}
It may be inferred from the definition that the number of $1$-boxes in the first row (of any marginally large tableau) is larger than the total number of boxes appearing in the second row by exactly one.
Furthermore, the second row should contain exactly one $2$-box.
\end{remark}

\begin{remark}
Alternatively, one could define a marginally large tableau to be an element of~$T(\la)$, for some $\la\in P^+$, that contains exactly one basic $i$-column, for each $i\in I$.
Since the tableau description $T(\la)$ for the $G_2$ type is a well-known result~\cite{KaMi}, we refer readers to the original papers and shall not repeat the complete definition here.
\end{remark}

The set $\SML$ consists of all tableaux of the following form.
\begin{equation*}
T = \raisebox{-0.5\height}{\ %
\begin{texdraw}
 \fontsize{6}{6}\selectfont
 \textref h:C v:C
 \drawdim em
 \setunitscale 1.35
 \move(14.5 0)\lvec(1 0)\lvec(1 -1)\lvec(14.5 -1)\lvec(14.5 0)
 \ifill f:0.8
 \move(0 0)\lvec(-2.5 0)\lvec(-2.5 -2)\lvec(0 -2)\lvec(0 0)
 \ifill f:0.8
 \move(14.5 0)\lvec(-3.5 0)
 \move(14.5 -1)\lvec(-3.5 -1)
 \move(0 -2)\lvec(-3.5 -2)
 \move(-3.5 0)\lvec(-3.5 -2)
 \htext(-3 -0.5){$1$}
 \htext(-0.55 -0.5){$1$}
 \htext(-1.7 -0.5){$\cdots$}
 \move(0 0)\rlvec(0 -2)
 \htext(0.5 -0.5){$1$}
 \move(1 0)\rlvec(0 -1)
 \htext(2.25 -0.5){$2\!\cdots\!2$}
 \move(3.5 0)\rlvec(0 -1)
 \htext(4.75 -0.5){$3\!\cdots\!3$}
 \move(6 0)\rlvec(0 -1)
 \htext(6.5 -0.5){$0$}
 \move(7 0)\rlvec(0 -1)
 \htext(8.25 -0.5){$\bar{3}\!\cdots\!\bar{3}$}
 \move(9.5 0)\rlvec(0 -1)
 \htext(10.75 -0.5){$\bar{2}\!\cdots\!\bar{2}$}
 \move(12 0)\rlvec(0 -1)
 \htext(13.25 -0.5){$\bar{1}\!\cdots\!\bar{1}$}
 \move(14.5 0)\rlvec(0 -1)
 \htext(-1.25 -1.5){$3\!\cdots\!3$}
 \move(-2.5 -0.9)\rlvec(0 -1.1)
 \htext(-3 -1.5){$2$}
\end{texdraw}%
}
\end{equation*}
The non-shaded parts, which are basic columns, must exist, whereas the shaded parts are optional and, except for the $0$-box, can be of arbitrary widths.
The simplest marginally large tableau is
\begin{equation*}
\hwvT=\raisebox{-0.4\height}{\ %
\begin{texdraw}%
\fontsize{6}{6}\selectfont
\textref h:C v:C \drawdim em
\setunitscale 1.35
\move(0 2)\rlvec(2 0)
\move(0 1)\rlvec(2 0)\rlvec(0 1)
\move(0 0)\rlvec(1 0)\rlvec(0 2)
\move(0 0)\rlvec(0 2)
\htext(0.5 1.5){$1$} \htext(1.5 1.5){$1$}
\htext(0.5 0.5){$2$}
\end{texdraw}
},
\end{equation*}
which consists of just the basic columns.
It corresponds to the highest weight element $b_{\infty}\in \Binf$.

\begin{theorem}
The set $\SML$ of all marginally large tableaux forms a crystal and is isomorphic to the crystal $\Binf$.
\end{theorem}

Let us recall the crystal structure on the set $T(\infty)$.
We start with the description of the Kashiwara operator actions.
As a first step, we recall the operator actions on the crystal~$T(\la)$.
One first expands a given tableau into its \emph{tensor product form} through the \emph{far eastern reading}, as exemplified by
\begin{equation*}
\begin{tabular}{rll}
\raisebox{-0.4\height}{
 \begin{texdraw}%
 \drawdim in
 \fontsize{6}{6}\selectfont
 \textref h:C v:C
 \drawdim em
 \setunitscale 1.35
 \move(0 2)\rlvec(7 0)
 \move(0 1)\rlvec(7 0)\rlvec(0 1)
 \move(0 0)\rlvec(2 0)\rlvec(0 2)
 \move(0 0)\rlvec(0 2)
 \move(1 0)\rlvec(0 2)
 \move(3 1)\rlvec(0 1)
 \move(6 1)\rlvec(0 1)
 \move(7 1)\rlvec(0 1)
 \move(5 1)\rlvec(0 1)
 \move(4 1)\rlvec(0 1)
 \htext(0.5 1.5){$1$}
 \htext(1.5 1.5){$1$}
 \htext(2.5 1.5){$1$}
 \htext(6.5 1.5){$\bar 1$}
 \htext(5.5 1.5){$\bar 3$}
 \htext(4.5 1.5){$0$}
 \htext(3.5 1.5){$3$}
 \htext(0.5 0.5){$2$}
 \htext(1.5 0.5){$3$}
 \end{texdraw}%
 }
 \ \ =
\raisebox{-0.3\height}{
\begin{texdraw}%
\fontsize{5}{5}\selectfont
\textref h:C v:C
\drawdim em
\setunitscale 1.5
\move(0 0)\rlvec(1 0)\rlvec(0 1)\rlvec(-1 0)\rlvec(0 -1)
\htext(0.5 0.5){$\bar 1$}
\end{texdraw}%
}
$\otimes$
\raisebox{-0.3\height}{%
\begin{texdraw}%
\fontsize{5}{5}\selectfont
\textref h:C v:C
\drawdim em
\setunitscale 1.5
\move(0 0)\rlvec(1 0)\rlvec(0 1)\rlvec(-1 0)\rlvec(0 -1)
\htext(0.5 0.5){$\bar 3$}
\end{texdraw}%
}
$\otimes$
\raisebox{-0.3\height}{%
\begin{texdraw}%
\fontsize{6}{6}\selectfont
\textref h:C v:C
\drawdim em
\setunitscale 1.5
\move(0 0)\rlvec(1 0)\rlvec(0 1)\rlvec(-1 0)\rlvec(0 -1)
\htext(0.5 0.5){$0$}
\end{texdraw}%
}
$\otimes$
\raisebox{-0.3\height}{%
\begin{texdraw}%
\fontsize{6}{6}\selectfont
\textref h:C v:C
\drawdim em
\setunitscale 1.5
\move(0 0)\rlvec(1 0)\rlvec(0 1)\rlvec(-1 0)\rlvec(0 -1)
\htext(0.5 0.5){$3$}
\end{texdraw}%
}
$\otimes$
\raisebox{-0.3\height}{%
\begin{texdraw}%
\fontsize{6}{6}\selectfont
\textref h:C v:C
\drawdim em
\setunitscale 1.5
\move(0 0)\rlvec(1 0)\rlvec(0 1)\rlvec(-1 0)\rlvec(0 -1)
\htext(0.5 0.5){$1$}
\end{texdraw}%
}
$\otimes$
\raisebox{-0.3\height}{%
\begin{texdraw}%
\fontsize{6}{6}\selectfont
\textref h:C v:C
\drawdim em
\setunitscale 1.5
\move(0 0)\rlvec(1 0)\rlvec(0 1)\rlvec(-1 0)\rlvec(0 -1)
\htext(0.5 0.5){$1$}
\end{texdraw}%
}
$\otimes$
\raisebox{-0.3\height}{%
\begin{texdraw}%
\fontsize{5}{5}\selectfont
\textref h:C v:C
\drawdim em
\setunitscale 1.5
\move(0 0)\rlvec(1 0)\rlvec(0 1)\rlvec(-1 0)\rlvec(0 -1)
\htext(0.5 0.5){$3$}
\end{texdraw}%
}
$\otimes$
\raisebox{-0.3\height}{%
\begin{texdraw}%
\fontsize{5}{5}\selectfont
\textref h:C v:C
\drawdim em
\setunitscale 1.5
\move(0 0)\rlvec(1 0)\rlvec(0 1)\rlvec(-1 0)\rlvec(0 -1)
\htext(0.5 0.5){$1$}
\end{texdraw}%
}
$\otimes$
\raisebox{-0.3\height}{%
\begin{texdraw}%
\fontsize{6}{6}\selectfont
\textref h:C v:C
\drawdim em
\setunitscale 1.5
\move(0 0)\rlvec(1 0)\rlvec(0 1)\rlvec(-1 0)\rlvec(0 -1)
\htext(0.5 0.5){$2$}
\end{texdraw}%
}\,.
\end{tabular}
\end{equation*}
Then, the \emph{tensor product rule} is used to apply $\fit$ or $\eit$ to one of the boxes, after which the resulting tensor product form is reconstructed into the shape of the original tableau.

Now, to apply $\fit$ to a marginally large tableau, we go through the following procedure.
\begin{enumerate}
\item
Apply $\fit$ to the tableau as usual.
If the result is a large tableau, it is automatically marginally large, and we are done.
\item
If the result is not large, then $\fit$ must have been applied to the $i$-box in the basic $i$-column.
Insert one basic $i$-column to the left of the box $\fit$ acted on.
\end{enumerate}
Analogous process for the $\eit$ operator is as follows.
\begin{enumerate}
\item Apply $\eit$ to the tableau as usual. If the result is zero or a marginally large tableau, we are done.
\item Otherwise, the result is large, but not marginally large.
   The $\eit$ operator must have acted on the box sitting to the right of the $i$-box in the basic $i$-column.
   Remove the column containing the changed box, which must be a new basic $i$-column.
\end{enumerate}

We illustrate the $\tilde{f}_2$ action on $\hwvT$ and the $\tilde{f}_1$ action on $\tilde{f}_2(\hwvT)$.
The dark shaded boxes are the ones $\fit$ has acted on, and the light shadings show columns inserted to preserve largeness.
\begin{align*}
\hwvT= \raisebox{-0.4\height}{\ %
\begin{texdraw}%
\fontsize{6}{6}\selectfont
\textref h:C v:C \drawdim em
\setunitscale 1.35
\move(0 2)\rlvec(2 0)
\move(0 1)\rlvec(2 0)\rlvec(0 1)
\move(0 0)\rlvec(1 0)\rlvec(0 2)
\move(0 0)\rlvec(0 2)
\htext(0.5 1.5){$1$} \htext(1.5 1.5){$1$}
\htext(0.5 0.5){$2$}
\end{texdraw}
}
\dashrightarrow
\raisebox{-0.4\height}{\ %
\begin{texdraw}
 \fontsize{6}{6}\selectfont
 \textref h:C v:C
 \drawdim em
 \setunitscale 1.35
 \move(0 0)\rlvec(1 0)\rlvec(0 1)\rlvec(-1 0)\ifill f:0.5
 \move(0 2)\rlvec(2 0)
 \move(0 1)\rlvec(2 0)\rlvec(0 1)
 \move(0 0)\rlvec(1 0)\rlvec(0 2)
 \move(0 0)\rlvec(0 2)
 \htext(0.5 1.5){$1$} \htext(1.5 1.5){$1$}
 \htext(0.5 0.5){$3$}
\end{texdraw}
}
\dashrightarrow &
\raisebox{-0.4\height}{\ %
\begin{texdraw}
 \fontsize{6}{6}\selectfont
 \textref h:C v:C
 \drawdim em
 \setunitscale 1.35
 \move(0 0)\rlvec(1 0)\rlvec(0 2)\rlvec(-1 0)\ifill f:0.8
 %\move(1 1)\rlvec(1 0)\rlvec(0 1)\rlvec(-1 0)\ifill f:0.5
 \move(0 2)\rlvec(3 0)
 \move(1 0)\rlvec(0 2)
 \move(0 1)\rlvec(3 0)\rlvec(0 1)
 \move(0 0)\rlvec(2 0)\rlvec(0 2)
 \move(0 0)\rlvec(0 2)
 \htext(0.5 1.5){$1$} \htext(1.5 1.5){$1$} \htext(2.5 1.5){$1$}
 \htext(0.5 0.5){$2$} \htext(1.5 0.5){$3$}
\end{texdraw}
}=\tilde{f}_2(\hwvT)\\
&\dashrightarrow
\raisebox{-0.4\height}{\ %
\begin{texdraw}
 \fontsize{6}{6}\selectfont
 \textref h:C v:C
 \drawdim em
 \setunitscale 1.35
 \move(2 1)\rlvec(1 0)\rlvec(0 1)\rlvec(-1 0)\ifill f:0.5
 \move(0 2)\rlvec(3 0)
 \move(1 0)\rlvec(0 2)
 \move(0 1)\rlvec(3 0)\rlvec(0 1)
 \move(0 0)\rlvec(2 0)\rlvec(0 2)
 \move(0 0)\rlvec(0 2)
 \htext(0.5 1.5){$1$} \htext(1.5 1.5){$1$} \htext(2.5 1.5){$2$}
 \htext(0.5 0.5){$2$} \htext(1.5 0.5){$3$}
\end{texdraw}
}
\dashrightarrow
\raisebox{-0.4\height}{\ %
\begin{texdraw}
 \fontsize{5}{5}\selectfont
 \textref h:C v:C
 \drawdim em
 \setunitscale 1.35
 \move(2 1)\rlvec(1 0)\rlvec(0 1)\rlvec(-1 0)\ifill f:0.8
 \move(0 2)\rlvec(4 0)
 \move(1 0)\rlvec(0 2)
 \move(3 1)\rlvec(0 1)
 \move(0 1)\rlvec(4 0)\rlvec(0 1)
 \move(0 0)\rlvec(2 0)\rlvec(0 2)
 \move(0 0)\rlvec(0 2)
 \htext(0.5 1.5){$1$} \htext(1.5 1.5){$1$} \htext(2.5 1.5){$1$} \htext(3.5 1.5){$2$}
 \htext(0.5 0.5){$2$} \htext(1.5 0.5){$3$}
\end{texdraw}
}=\tilde{f}_1(\tilde{f}_2(\hwvT))
\end{align*}

Any given marginally large tableau $T\in T(\infty)$ is contained in exactly one $\Sla$ for some $\la\in\pwl$.
The weight~$\wt(T)$ is $\la$ less than the weight of tableau~$T$ viewed as an element of the crystal $\Sla$.
Furthermore, $\veps_i(T)$ is set to the corresponding value computed for $T$ as an element of $\Sla$, and $\vphi_i(T)=\veps_i(T)+\wt(T)(h_i)$.

%%%%%%%%%%%%%%%%%%%%%%%%%%%%%%%%%%%%%%%%%%%%%%%%%%%%%%%%%%%%%
\section{Monomial Description of crystal $\Bla$}\label{sec:3}

The monomial $N_{\la}=Y_1(1)^{\la(h_1)}Y_2(0)^{\la(h_2)}$ is a maximal element of the crystal $\NM$ of weight $\la=\la(h_1)\La_1+\la(h_2)\La_2$.
In this section, we present the connected component of~$\NM$, containing the maximal vector $N_{\la}$.
The previous work~\cite{sh} also gave a monomial description of the crystal $\Bla$ by locating as similar connected component.
However, our description will be completely different and more suitable for our later needs.
One characteristic of our new description is that the expression for each element displays how many of $\tilde{f}_1$ and $\tilde{f}_2$ need to be used to reach the element from the highest weight element.

For each $\la\in\pwl$, let $\Nla$ be the set of all monomials of the form
\begin{equation}\label{def:3-1}
\begin{aligned}
Y_1(1)^{\la(h_1)}Y_2(0)^{\la(h_2)} \cdot\,
&U_2(0)^{-u_{2,0}}U_1(1)^{-u_{1,1}}\\
&U_2(1)^{-u_{2,1}}U_1(2)^{-u_{1,2}}\\
&U_2(2)^{-u_{2,2}}U_1(3)^{-u_{1,3}}
\end{aligned}
\end{equation}
with the exponents satisfying
\begin{equation}\label{deficond}
\begin{aligned}
&0 \leqslant u_{2,0} \leqslant \la(h_2),
& &0 \leqslant u_{1,1} \leqslant 3u_{2,0}\!+\!\la(h_1),\\
&0 \leqslant u_{2,1} \leqslant \min\{\frac{2u_{1,1}\!+\!\la(h_1)}{3},u_{1,1}\},
& &0 \leqslant u_{1,2} \leqslant \min\{\frac{3u_{2,1}\!+\!\la(h_1)}{2},2u_{2,1}\},\\
&0 \leqslant u_{2,2} \leqslant \min\{\frac{u_{1,2}\!+\!\la(h_1)}{3},\frac{u_{1,2}}{2}\},
& &0 \leqslant u_{1,3} \leqslant \min\{\la(h_1),u_{2,2}\}.
\end{aligned}
\end{equation}
The vector $N_{\la}=Y_1(1)^{\la(h_1)}Y_2(0)^{\la(h_2)}$ is contained in this set $\Nla$, as can be seen by taking $u_{i,m}=0$, for every~$i$ and~$m$.

\begin{remark}
In terms of the $Y_i(m)$s, the monomial of~\eqref{def:3-1} may be expressed as follows:
\begin{equation}\label{def:3-2}
\begin{aligned}
&Y_1(1)^{\la(h_1)+3u_{2,0}-u_{1,1}}
Y_1(2)^{3u_{2,1}-u_{1,1}-u_{1,2}}
Y_1(3)^{3u_{2,2}-u_{1,2}-u_{1,3}}
Y_1(4)^{-u_{1,3}}\\
&\cdot\, Y_2(0)^{\la(h_2)-u_{2,0}}
Y_2(1)^{-u_{2,0}-u_{2,1}+u_{1,1}}
Y_2(2)^{-u_{2,1}-u_{2,2}+u_{1,2}}
Y_2(3)^{-u_{2,2}+u_{1,3}}.
\end{aligned}
\end{equation}
\end{remark}

\begin{theorem}\label{prop:3}
The set $\Nla$ is the connected component of the crystal $\NM$, containing the vector $N_{\la}=Y_1(1)^{\la(h_1)}Y_2(0)^{\la(h_2)}$ of weight $\la$, and $\Nla\cong \Bla$.
\end{theorem}
\begin{proof}
Since the final claim follows from Theorem~\ref{nakakash}, it suffices to show that the actions of the Kashiwara operators on $\Nla$ satisfy the properties
\begin{equation*}
\fit \Nla\subset \Nla\cup\{0\}, \quad \eit \Nla\subset \Nla\cup\{0\},
\end{equation*}
for all $i\in I$, and that every element of $\Nla$ is connected to the element $N_{\la}$ by Kashiwara operator actions.

Fix $i\in I=\{1,2\}$ and let us suppose $N\in \Nla$ is such that $\fit N \not\in \Nla\cup \{0\}$.
Then, since $\fit N \neq 0$, we must have $\vphi_i(N)>0$ and $\fit N=U_i(m)^{-1}N$, for some $m$.
Here, $1\leqslant m\leqslant 3$ for the $i=1$ case, and $0\leqslant m\leqslant 2$ for the $i=2$ case, as can be seen from the general form~\eqref{def:3-2} of the monomial in~$N(\la)$ and the definition of $m_f(N,i)$ given by~\eqref{mfme}.
Below, we will show that the assumption on $N$ leads to contradictions, for such $i$ and~$m$.

(1) Let us first assume $i=1$ and $m=1$ were used.
Then, the fact $m_f(N,1)=m=1$ implies that the exponent $y_1(1)=\la(h_1)+3u_{2,0}-u_{1,1}$ of $Y_1(1)$ appearing in the monomial $N$ is positive, so that we have
\begin{equation*}
\la(h_1)+3u_{2,0}>u_{1,1}.
\end{equation*}
On the other hand, since the assumption $\fit N \not\in \Nla$ can only be associated with the violation of the restriction for the exponent of $U_1(1)$ appearing in the monomial $\fit N$, given in the condition~\eqref{deficond}, we can state
\begin{equation*}
u_{1,1}+1> \la(h_1)+3u_{2,0}.
\end{equation*}
Since these two inequalities cannot be true simultaneously, we have a contradiction.

(2) Assume $i=1$ and $m=2$ were used.
The fact $m_f(N,1)=2$ implies $y_1(2)>0$, so that we have
\begin{equation*}
3u_{2,1}-u_{1,1}> u_{1,2}.
\end{equation*}
On the other hand, assumption $\fit N \not\in \Nla$ implies
\begin{equation*}
u_{1,2}+1>\min\{\frac{3u_{2,1}+\la(h_1)}{2},2u_{2,1}\}.
\end{equation*}
A combination of the two inequalities implies
\begin{equation*}
3u_{2,1}-u_{1,1} > \min\{\frac{3u_{2,1}+\la(h_1)}{2},2u_{2,1}\},
\end{equation*}
so that at least one of
\begin{align*}
u_{2,1} > \frac{2u_{1,1}+\la(h_1)}{3},\quad
u_{2,1}> u_{1,1}
\end{align*}
must be true.
However, both of these contradict the assumption $N\in \Nla$.

(3) Assume $i=1$ and $m=3$ were used.
Since $m_f(N,1)=3$ implies $y_1(3)>0$, we have
\begin{equation*}
3u_{2,2}-u_{1,2}>u_{1,3}.
\end{equation*}
But, assumption $\fit N \not\in \Nla$ implies
\begin{equation*}
u_{1,3}+1>\min\{\la(h_1),u_{2,2}\}.
\end{equation*}
A combination of the two inequalities implies that at least one of
\begin{align*}
u_{2,2} > \frac{u_{1,2}+\la(h_1)}{3},\quad
u_{2,2} > \frac{u_{1,2}}{2}
\end{align*}
must be true.
However, both of these contradict the assumption $N\in \Nla$.

(4) Let us next treat the $i=2$ and $m=0$ case.
As before, the process of obtaining $m_f(N,2)=0$ implies that $y_2(0)>0$, so that we have
\begin{equation*}
\la(h_2)>u_{2,0}.
\end{equation*}
But, since the assumption $\fit N \not\in \Nla$ can only be associated with the violation of the first line of the condition~\eqref{deficond}, we can state
\begin{equation*}
u_{2,0}+1 > \la(h_2).
\end{equation*}
Since these two inequalities cannot be true simultaneously, we have a contradiction.

(5) Assume $i=2$, and $m=1$ or $2$ were used.
Let us briefly use the notation $A$, $B$, and $C$ as follows:
\begin{center}
      \begin{tabular}{r|c|c|c}
      $m$ & A & B & C\\
      \hline
      \hline
      $1$ & $-u_{2,0}+u_{1,1}$ & $\frac{2u_{1,1}+\la(h_1)}{3}$ & $u_{1,1}$\\
      $2$ & $-u_{2,1}+u_{1,2}$ & $\frac{u_{1,2}+\la(h_1)}{3}$ & $\frac{u_{1,2}}{2}$
      \end{tabular}
\end{center}
The fact $m_f(N,2)=m$ implies that $y_2(m)>0$, so that we have
\begin{equation*}
A>u_{2,m}.
\end{equation*}
On the other hand, assumption $\fit N \not\in \Nla$ implies
\begin{equation*}
u_{2,m}+1>\min\{B,C\}.
\end{equation*}
A combination of the two inequalities implies
\begin{equation*}
A > \min\{B,C\},
\end{equation*}
so that at least one of $A>B$, $A>C$ must be true.
However, both of these contradict the assumption $N\in \Nla$.

This completes the proof that $\fit \Nla\subset \Nla\cup\{0\}$.
Verification of $\eit \Nla\subset \Nla\cup\{0\}$ can be done similarly.

To show the connectedness of $\Nla$, it suffices to show that the only maximal element in $\Nla$ is $N_{\la}=Y_1(1)^{\la(h_1)}Y_2(0)^{\la(h_2)}$.
Suppose $N\in \Nla$ is such that $\eit(N)=0$, for all $i\in I$, and suppose $N\neq N_{\la}$.
The latter assumption ensures the existence for at least one positive $u_{i,m}$ associate to $N$.
We first locate the largest~$m$ for which there is a positive $u_{i,m}$ and then choose the largest~$i$ for which $u_{i,m}$ is positive, with the~$m$ already fixed.
For such $i$ and $m$, we can confirm
\begin{equation*}
y_i(m+1)=-u_{i,m}<0
\end{equation*}
and that $y_i(k)=0$ for all $k>m+1$, using the exponents of $Y_i(k)$ appearing in the expression~\eqref{def:3-2} of a monomial $N$.
This implies
\begin{equation*}
-\sum_{k>m} y_i(k)>0,
\end{equation*}
so that
\begin{equation*}
\veps_i(N)=\max\big\{-\sum_{k>j} y_i(k)\,|\, j\in\Z\big\}>0.
\end{equation*}
This contradicts the assumption of $N$ being a maximal vector.
\end{proof}

The highest weight element of the crystal $\Nla$ is $N_{\la} = Y_1(1)^{\la(h_1)}Y_2(0)^{\la(h_2)}$.
The element given by~\eqref{def:3-1} can be reached from $N_{\la}$ through applications of $(u_{1,1}+u_{1,2}+u_{1,3})$-many $\tilde{f}_1$ and $(u_{2,0}+u_{2,1}+u_{2,2})$-many $\tilde{f}_2$, in some order, and so the weight of the element is
\begin{equation*}
\la-(u_{1,1}+u_{1,2}+u_{1,3})\alpha_1-(u_{2,0}+u_{2,1}+u_{2,2})\alpha_2.
\end{equation*}

%%%%%%%%%%%%%%%%%%%%%%%%%%%%%%%%%%%%%%%%%%%%%%%%%%%%%%%%%%%%%%%%%%%%%%%%%%%%%%%%%%%%%%%%%%
\section{Crystal $\Bla$ in $\Binf$ expressed as Tableaux}\label{sec:4}

In this section, we present the crystal $\Bla$ using marginally large tableaux, which are elements from the tableau description $T(\infty)$ of~$\Binf$ recalled in Section~\ref{sec:2}.
The starting point in obtaining the goal of this section is the unique strict crystal embedding
\begin{equation*}
\Omega_{\la}: \Bla \hookrightarrow \Binf\otimes \Tvla,
\end{equation*}
such that the highest weight vector $b_{\la}$ is sent to $b_{\infty}\otimes \tvla$, where $b_{\infty}$ is the highest weight vector of $B(\infty)$~\cite{Naka}.
Here, $\la\in\pwl$ and the crystal $\Tvla$ is the single-element set $\{\tvla\}$ with the following crystal structure:
\begin{equation*}
\wt(\tvla)=\la,\quad
\veps_i(\tvla)=-\la(h_i),\quad
\vphi_i(\tvla)=0,\quad
\eit(\tvla)=0,\quad
\fit(\tvla)=0.
\end{equation*}
Since the image of the embedding $\Omega_{\la}$ is the connected component in the crystal $\Binf\otimes \Tvla$ containing the element $b_{\infty}\otimes \tvla$, our goal will be achieved by finding the connected component containing $\hwvT\otimes \tvla$ in $\SML\otimes \Tvla$, where $\hwvT$ denote the highest weight element of~$\SML$.

Let us set up a notation for handling the number of boxes appearing in a tableau.
For any given marginally large tableau, we define integers $\ut_{v,w}$ as follows.
\begin{enumerate}
\item Symbol $\ut_{2,3}$ is the number of all boxes labeled $3$, appearing on the second row of the tableau.
\item For each $w=2,3,\bar{2},\bar{1}$, define $\ut_{1,w}$ to be the combined number of all boxes labeled $w$ through $\bar{1}$, appearing on the first row of the tableau.
\item Define $\ut_{1,0}$ to be the combined number of all boxes labeled $0$ through $\bar{1}$ plus the combined number of all boxes labeled $\bar{3}$ through $\bar{1}$, appearing on the top row of the tableau. In other words, the $0$-box is counted once, and the boxes labeled $\bar{3}$ trough $\bar{1}$ are counted twice.
\end{enumerate}
We will not make the dependence of~$\ut_{v,w}$ on the tableau explicit, as our use of this notation will always be in such a way that the tableau under consideration is unambiguous.

\begin{remark}
The first row of a marginally large tableau consists of
\begin{equation}
\begin{aligned}
&\textup{$(\ut_{1,2}-\ut_{1,3})$-many $2$s},&
&\textup{$\lfloor A \rfloor$-many $3$s},\\
&\textup{$((A-\lfloor A\rfloor)+(B-\lfloor B\rfloor))$-many $0$s},&
&\textup{$\lfloor B\rfloor$-many $\bar{3}$s},\\
&\textup{$(\ut_{1,\bar{2}}-\ut_{1,\bar{1}})$-many $\bar{2}$s},&
&\textup{$\ut_{1,\bar{1}}$-many $\bar{1}$s},
\end{aligned}
\end{equation}
where $A=\ut_{1,3}-\frac{\ut_{1,0}}{2}$ and $B=\frac{\ut_{1,0}}{2}-\ut_{1,\bar{2}}$.
\end{remark}

\begin{remark}\label{rem}
A system of $\ut_{v,w}$'s uniquely identify an element of $\SML$.
That is, given a system of $\ut_{v,w}$'s a corresponding marginally large tableau may not exist unless all inequalities
\begin{equation}\label{rem-}
\ut_{1,\bar{1}} \leqslant \ut_{1,\bar{2}} \leqslant \frac{\ut_{1,0}}{2} \leqslant \ut_{1,3} \leqslant \ut_{1,2}
\end{equation}
are satisfied, but there can be at most one marginally large tableau that correspond to the $\ut_{v,w}$'s.
\end{remark}

For each $\la\in\pwl$, let $\SML^{\la}$ be the set of all marginally large tableaux, such that the $\ut_{v,w}$'s satisfy the condition
\begin{equation}\label{deficond3}
\begin{aligned}
&0 \leqslant \ut_{2,3} \leqslant \la(h_2),\quad &
&0 \leqslant \ut_{1,2} \leqslant \la(h_1)+3\ut_{2,3},\\
&0 \leqslant \ut_{1,3} \leqslant \frac{\la(h_1)+2\ut_{1,2}}{3},\quad &
&0 \leqslant \ut_{1,0} \leqslant \frac{\la(h_1)+3\ut_{1,3}}{2},\\
&0 \leqslant \ut_{1,\bar{2}} \leqslant \frac{\la(h_1)+\ut_{1,0}}{3},\quad &
&0 \leqslant \ut_{1,\bar{1}} \leqslant \la(h_1),
\end{aligned}
\end{equation}
and we also define the set $\SML_{\la} = \{ T\otimes \tvla \mid T\in \SML^{\la}\}$.

\begin{example}
An element from the set $\SML^{\la}$ takes the form
\begin{equation*}
T = \raisebox{-0.5\height}{\ %
\begin{texdraw}
 \fontsize{6}{6}\selectfont
 \textref h:C v:C
 \drawdim em
 \setunitscale 1.5
 \move(14.5 0)\lvec(1 0)\lvec(1 -1)\lvec(14.5 -1)\lvec(14.5 0)
 \ifill f:0.8
 \move(0 -1)\lvec(-2.5 -1)\lvec(-2.5 -2)\lvec(0 -2)\lvec(0 -1)
 \ifill f:0.8
 \move(14.5 0)\lvec(-3.5 0)
 \move(14.5 -1)\lvec(-3.5 -1)
 \move(0 -2)\lvec(-3.5 -2)
 \move(-3.5 0)\lvec(-3.5 -2)
 \htext(-3 -0.5){$1$}
 \htext(-0.55 -0.5){$1$}
 \htext(-1.7 -0.5){$\cdots$}
 \move(0 0)\rlvec(0 -2)
 \htext(0.5 -0.5){$1$}
 \move(1 0)\rlvec(0 -1)
 \htext(2.25 -0.5){$2\!\cdots\!2$}
 \move(3.5 0)\rlvec(0 -1)
 \htext(4.75 -0.5){$3\!\cdots\!3$}
 \move(6 0)\rlvec(0 -1)
 \htext(6.5 -0.5){$0$}
 \move(7 0)\rlvec(0 -1)
 \htext(8.25 -0.5){$\bar{3}\!\cdots\!\bar{3}$}
 \move(9.5 0)\rlvec(0 -1)
 \htext(10.75 -0.5){$\bar{2}\!\cdots\!\bar{2}$}
 \move(12 0)\rlvec(0 -1)
 \htext(13.25 -0.5){$\bar{1}\!\cdots\!\bar{1}$}
 \move(14.5 0)\rlvec(0 -1)
 \htext(-1.25 -1.5){$3\!\cdots\!3$}
 \move(-2.5 -0.9)\rlvec(0 -1.1)
 \htext(-3 -1.5){$2$}
\end{texdraw}}\,,
\end{equation*}
with the shaded parts sized so that the $\ut_{v,w}$'s satisfy the conditions of~\eqref{deficond3}.
The highest weight element
\begin{equation*}
\hwvT=\raisebox{-0.4\height}{\ %
\begin{texdraw}%
\fontsize{6}{6}\selectfont
\textref h:C v:C \drawdim em
\setunitscale 1.5
\move(0 2)\rlvec(2 0)
\move(0 1)\rlvec(2 0)\rlvec(0 1)
\move(0 0)\rlvec(1 0)\rlvec(0 2)
\move(0 0)\rlvec(0 2)
\htext(0.5 1.5){$1$} \htext(1.5 1.5){$1$}
\htext(0.5 0.5){$2$}
\end{texdraw}
}
\end{equation*}
of $\SML$ belongs to the set $\SML^{\la}$, for any $\la\in\pwl$, with all its associated $\ut_{v,w}=0$.
\end{example}

\begin{lemma}\label{lem}
The actions of the Kashiwara operators on $\SML_{\la}$ satisfy the properties
\begin{equation*}
\fit \SML_{\la}\subset \SML_{\la}\cup\{0\},\quad \eit \SML_{\la}\subset \SML_{\la}\cup\{0\},
\end{equation*}
for all $i\in I$.
\end{lemma}
\begin{proof}
We will just deal with the $\fit$ case, since the $\eit$ case is very similar.
The tensor product rule states that
\begin{equation*}
\fit (T\otimes \tvla)=
\begin{cases}
T \otimes \fit(\tvla) &\text{when}\ \vphi_i(T)\leqslant \veps_i(\tvla),\\
\fit(T) \otimes \tvla &\text{when}\ \vphi_i(T)> \veps_i(\tvla),
\end{cases}
\end{equation*}
and since $\fit(\tvla)=0$, we have $\fit(T\otimes \tvla)=0$, whenever $\vphi_i(T)\leqslant \veps_i(\tvla)$.
Hence, it suffices to show that $\fit(T)\in T(\infty)^{\la}$, for any $T\in T(\infty)^{\la}$ satisfying the condition $\vphi_i(T)> \veps_i(\tvla)$.

Note that for any tableau $T\in \SML^{\la}$, the Kashiwara operator $\tilde{f}_1$ must act on a $3$-box, a $0$-box, a $\bar{2}$-box, or the unique basic $1$-column, appearing on the first row of $T$.
The Kashiwara operator $\tilde{f}_2$ must act on a $2$-box, a $\bar{3}$-box, appearing on the first row, or the $2$-box in the unique basic $2$-column of $T$.

Fix any $i\in I$, and suppose that $T\otimes \tvla\in \SML_{\la}$ is such that
\begin{equation}\label{4-111}
\vphi_i(T)> \veps_i(\tvla).
\end{equation}
Since we know that $\fit (T)$ is a marginally large tableau, it is suffices to show that it satisfies the condition~\eqref{deficond3}.

(1) Suppose that $i=1$ and $\tilde{f}_1$ has acted on an $1$-box which is the unique basic $1$-column of~$T$.
In this case, a single basic $1$-column must have been inserted into $T$ during the $\fit$ action.
The assumption implies
\begin{align*}
\vphi_i(T) =3\ut_{2,3}-\ut_{1,2}.
\end{align*}
Applying this value on the inequality~\eqref{4-111}, we obtain
\begin{align*}
3\ut_{2,3}-\ut_{1,2}=\vphi_1(T) > \veps_1(\tvla)=-\la(h_1),
\end{align*}
and this means
\begin{align*}
\ut_{1,2}+1\leqslant \la(h_1)+3\ut_{2,3}.
\end{align*}
This shows that $\tilde{f}_1(T)$ belongs to $\SML^{\la}$.

(2) Suppose that $i=1$ and $\tilde{f}_i$ has acted on a $3$-box or $0$-box in the first row of~$T$.
Then this assumption and the tensor product rule on multiple tensors imply
\begin{align*}
\ut_{1,2}-\ut_{1,3} < 2\ut_{1,3}-\ut_{1,0}.
\end{align*}
On the other hand, since $T$ satisfies the condition~\eqref{deficond3}, we have
\begin{align*}
\ut_{1,3} \leqslant \frac{\la(h_{1})+2\ut_{1,2}}{3}.
\end{align*}
A combination of the two inequalities implies
\begin{align*}
\ut_{1,0}+1 \leqslant \frac{\la(h_{1})+3\ut_{1,3}}{2}.
\end{align*}
Thus, $\tilde{f}_1(T)$ is contained in $\SML^{\la}$.

(3) Suppose that $i=1$ and $\tilde{f}_i$ has acted on a $\bar{2}$-box in the top row of~$T$.
This assumption and the tensor product rule imply
\begin{align*}
\ut_{1,0}-2\ut_{1,\bar{2}} < \ut_{1,\bar{2}}- \ut_{1,\bar{1}}.
\end{align*}
On the other hand, since $T\in \SML^{\la}$, we have
\begin{align*}
\ut_{1,\bar{2}} \leqslant \frac{\la(h_{1})+\ut_{1,0}}{3}.
\end{align*}
A combination of the two inequalities implies
\begin{align*}
\ut_{1,\bar{1}}+1 \leqslant \la(h_{1}).
\end{align*}
Thus, $\tilde{f}_1(T)$ is contained in $\SML^{\la}$.

(4) Suppose that $i=2$ and $\tilde{f}_2$ has acted on an $2$-box that is contained in the unique basic $2$-column of~$T$.
In this case, a single basic $2$-column must have been inserted into $T$ during the $\fit$ action.
The assumption implies
\begin{align*}
\vphi_i(T) =-\ut_{2,3},
\end{align*}
and so
\begin{align*}
-\ut_{2,3}=\vphi_2(T) >\veps_2(\tvla)=-\la(h_2),
\end{align*}
which follows from~\eqref{4-111}.
This means
\begin{align*}
\ut_{2,3}+1 \leqslant \la(h_2).
\end{align*}

(5) Suppose that $i=2$ and $\tilde{f}_2$ has acted on a $2$-box in the first row of $T$.
This assumption implies that
\begin{equation*}
t_{2,3}< t_{1,2}-t_{1,3}.
\end{equation*}
By combining this and the condition
\begin{equation*}
t_{1,2}\leqslant \la(h_1)+3t_{2,3}
\end{equation*}
in~\eqref{deficond3}, we obtain
\begin{equation*}
t_{1,3}+1\leqslant \frac{\la(h_1)+2t_{1,2}}{3}.
\end{equation*}
It is clearly that $\tilde{f}_1(T)$ is contained in $\SML^{\la}$.

(6) Suppose that $i=2$ and $\tilde{f}_i$ has acted on a $\bar{3}$-box in the first row of~$T$.
Then this assumption and the tensor product rule imply
\begin{align*}
\ut_{1,3}-\ut_{1,0}+\ut_{1,\bar{2}} < 0.
\end{align*}
And, since $T$ satisfies the condition~\eqref{deficond3}, we have
\begin{align*}
\ut_{1,0}\leqslant \frac{\la(h_{1})+3\ut_{1,3}}{2}.
\end{align*}
A combination of the two inequalities implies
\begin{align*}
\ut_{1,\bar{2}}+1 \leqslant \frac{\ut_{1,0}+\la(h_{1})}{3}.
\end{align*}

This completes the verification of $\fit \SML_{\la}\subset \SML_{\la}\cup\{0\}$.
\end{proof}

Now, in order to achieve our main gaol, it suffices to show that every element of $\SML_{\la}$ is connected to the element $T_{\infty}\otimes \tvla$.
This could be done directly, but we will use a slightly different approach.
We will compute the unique strict crystal embedding $\Omega_{\la}: \Bla \hookrightarrow \Binf\otimes \Tvla$ explicitly as maps between the crystal realizations $N(\la)$ and $T(\infty)\otimes \Tvla$.
By showing that the image of this map is the set $\SML_{\la}$, we will accomplish our goal and show the correspondence between the crystals $\Nla$ and $\SML_{\la}$, at the same time.

Reviewing the sets $\Nla$ and $\SML_{\la}$, focusing on the many inequality conditions, one can notice a similarity between the two general elements of $\Nla$ and $\SML_{\la}$, which leads to a natural correspondence between $\Nla$ and $\SML_{\la}$.
For any $\la\in\pwl$, we define the map
\begin{equation}\label{omega}
\omega_{\la}:\Nla\rightarrow \SML_{\la}\ (\subset T(\infty)\otimes \Tvla)
\end{equation}
to send the general monomial $N$ of~\eqref{def:3-1} satisfying~\eqref{deficond} to $T\otimes \tvla$, where $T$ is the unique marginally large tableau associated with the system of $t_{v,w}$'s given by
\begin{equation}\label{corr}
\ut_{2,3}=u_{2,0},\
\ut_{1,2}=u_{1,1},\
\ut_{1,3}=u_{2,1},\
\ut_{1,0}=u_{1,2},\
\ut_{1,\bar{2}}=u_{2,2},\
\ut_{1,\bar{1}}=u_{1,3}.
\end{equation}
Under the correspondence~\eqref{corr}, the conditions of~\eqref{deficond} imply that the conditions~\eqref{rem-} of Remark~\ref{rem} are satisfied, so that the existence of such a marginally large tableau is ensured.
The conditions of~\eqref{deficond} further imply that the tableau satisfies the conditions of~\eqref{deficond3}, so that the map is well defined.
This map sends the monomial $N_{\infty}$ to $T_{\infty}\otimes \tvla$, where $N_{\infty}$ and $T_{\infty}$ are the highest weight elements of crystals $N(\la)$ and $T(\infty)$, respectively, and it is easy to see that this map $\omega_{\la}$ is bijective.
The following is our main result.

\begin{theorem}
The set $\SML_{\la}$ is a subcrystal of $\SML \otimes \Tvla$, which is isomorphic to $\Nla\cong \Bla$.
\end{theorem}
\begin{proof}
Since the map $\omega_{\la}:\Nla\rightarrow T(\infty)\otimes \Tvla$ is injective, sends $N_{\infty}$ to $T_{\infty}\otimes \tvla$, and $\omega_{\la}(\Nla)=\SML_{\la}$, it suffices to show that the map $\omega_{\la}$ is a strict crystal morphism.

Let us assume a fixed $i\in I$ and take $\omega_{\la}(N)=T\otimes \tvla$, i.e., assume condition~\eqref{corr}, throughout this proof.
From the definitions of the maps $\wt$ and $\veps_i$ for the crystals $\Nla$ and $\SML\otimes \Tvla$, the equalities $\wt(N)=\wt(T\otimes \tvla)$ and $\veps_i(N)=\veps_i(T\otimes \tvla)$ are obtained through direct computation, under the identification~\eqref{corr}.
These two equalities also lead to the identity
\begin{equation*}
\vphi_i(N) =\veps_i(N)+\wt(N)(h_i) =\veps_i(T\otimes \tvla)+\wt(T\otimes \tvla)(h_i)=\vphi_i(T\otimes \tvla).
\end{equation*}
We will now focus our efforts in showing that the map $\omega_{\la}$ commutes with the Kashiwara operators $\fit$.
The $\eit$ part that needs to be checked may be done similarly.

Note that
\begin{equation*}
\vphi_i(T\otimes \tvla)
=\max\{\vphi_i(\tvla),\vphi_i(T)+\wt(\tvla)(h_i)\}
=\max\{0,\vphi_i(T)-\veps_i(\tvla)\}.
\end{equation*}
This implies
\begin{equation}\label{corr2}
\vphi_i(N)=\vphi_i(T\otimes \tvla)=\max\{0,\vphi_i(T)-\veps_i(\tvla)\}.
\end{equation}
Now, when $\fit N=0$, we have $\vphi_i(N)=0$ and so $\vphi_i(T) \leqslant \veps_i(\tvla)$ from~\eqref{corr2}, so that by the tensor product rule
\begin{equation*}
\fit (T\otimes \tvla) = T\otimes \fit(\tvla)=0.
\end{equation*}
When $\fit N \neq 0$, we must have $\vphi_i(N) > 0$.
From this and~\eqref{corr2}, we have
\begin{equation*}
\vphi_i(T)-\veps_i(\tvla) = \vphi_i(N) > 0,
\end{equation*}
and then the tensor product rule implies
\begin{equation*}
\fit(T\otimes \tvla) = \fit(T) \otimes \tvla.
\end{equation*}
On the other hand, we recall that when $\fit N \neq 0$, $\fit N=U_i(m)^{-1}N$ for some $m$.
Here, $1\leqslant m\leqslant 3$ when $i=1$, and $0\leqslant m\leqslant 2$ when $i=2$.
In each $i$ and $m$ case, $\tilde f_i$ acts on the $w$-box in the $v$th row of~$T$, as given below:
%\begin{table}
\begin{center}
      \begin{tabular}{c||c c c|c c c}
      $i$  &     & $1$         &           &     & $2$ &         \\
      \hline
      $m$  & $1$ & $2$         & $3$       & $0$ & $1$ & $2$     \\
      \hline
      \hline
      $v$  & $1$ & $1$         & $1$       & $2$ & $1$ & $1$\\
      \hline
      $w$  & $1$ & $3/0$       & $\bar{2}$ & $2$ & $2$ & $\bar{3}$\\
      \hline
      $w'$ & $2$ & $0/\bar{3}$ & $\bar{1}$ & $3$ & $3$ & $\bar{2}$
      \end{tabular}
\end{center}
In any case, the only (effective) difference between $T$ and $\fit T$ is that the combined number of all boxes labeled $w'$ through $\bar{1}$, appearing on their $v$th rows are $\ut_{v,w'}$ and $(\ut_{v,w'}+1)$, respectively, where $w'$ is as given above for each possible case.
Since we already know through Lemma~\ref{lem} that $\fit T$ belongs to~$\SML^{\la}$, $\fit(T) \otimes \tvla$ must be the image of $U_i(m)^{-1}N$ under $\omega_{\la}$.
We have shown that the map $\omega_{\la}$ commutes with the Kashiwara operator~$\fit$.
\end{proof}

Finally, we comment on another result concerning the description of $\Bla$ for the $G_2$ type.
Nakashima provided explicit descriptions of $\Bla$, for all rank two Lie algebra types, using the polyhedral realization approach~\cite{Naka}.
We will not provide any details, but it can be verified that when his results are specialized to the $G_2$~type, a resulting description of $\Bla$ is in natural correspondence with our description.

%%%%%%%%%%%%%%%%%%%%%%%%%%%

%%%%%%%%%%%%%%
\end{document}